\documentclass[a4paper,12pt]{article}
\usepackage{amsmath} \usepackage{amssymb} 
\usepackage{eucal}
\usepackage{geometry}
\usepackage{graphicx}
\usepackage[amsmath,thmmarks]{ntheorem}
\theoremstyle{change}
\theoremseparator{.}

\theorembodyfont{\normalfont\slshape}
\newtheorem{theorem}{Theorem}[section]
\newtheorem{corollary}[theorem]{Corollary}
\newtheorem{lemma}[theorem]{Lemma}
\newtheorem{proposition}[theorem]{Proposition}

\theorembodyfont{\normalfont}

\newtheorem{remark}[theorem]{Remark}

\def\proofsymbol{\rule{0.5em}{0.5em}}

\theoremheaderfont{\normalfont\itshape\bfseries}
\theoremstyle{nonumberplain}
\theoremsymbol{\enspace\ensuremath{{\proofsymbol}}}
\newtheorem{proof}{Proof}

\theoremstyle{empty}

\usepackage[shortlabels]{enumitem}
\setlist{
  listparindent=\parindent,
  parsep=0pt,
  itemsep=0.5em plus 0.25em minus 0.2em,
}

\setenumerate{
leftmargin=*,
labelsep=0.5em,
}
\setitemize{
leftmargin=*,
labelsep=0.5em,
}

\SetEnumerateShortLabel{s}{\textup{($\sigma$\arabic*)}}
\SetEnumerateShortLabel{a}{\textup{(\alph*)}}
\SetEnumerateShortLabel{A}{\textup{(\Alph*)}}
\SetEnumerateShortLabel{i}{\textup{(\roman*)}}
\SetEnumerateShortLabel{I}{\textup{(\Roman*)}}
\SetEnumerateShortLabel{m}{\textup{(m\arabic*)}}
\SetEnumerateShortLabel{d}{\textup{($\delta$\arabic*)}}
\SetEnumerateShortLabel{p}{\textup{(pm\arabic*)}}
\SetEnumerateShortLabel{1}{\textup{(\arabic*)}}
\SetEnumerateShortLabel{j}{\textup{(i\arabic*)}}
\SetEnumerateShortLabel{n}{\textup{(n\arabic*)}}

\newlist{lenumerate}{enumerate}{3}
\setlist[lenumerate]{
  listparindent=\parindent,
  itemsep=0.5em plus 0.25em minus 0.3em,
  fullwidth,
  labelsep=0.75em,
  label=\arabic*.
}

\numberwithin{equation}{section} 

\def\C{{\mathbb C}}

\def\N{{\mathbb N}}
\def\R{{\mathbb R}}

\def\CC{{\mathcal C}}

\def\CG{{\mathcal G}}

\def\CL{{\mathcal L}}

\def\CR{{\mathcal R}}
\def\CS{{\mathcal S}}

\def\phi{\varphi}

\def\qand{{\quad\mbox{and}\quad}}

\def\d{{\rm d}}

\def\1{{{\ae}}}
\def\2{{{\o}}}
\def\3{{{\aa}}}
\def\6{\, {\rm d}}

\def\ri{{\rm i}}
\def\e{{\rm e}}
\def\<{{\langle}}
\def\>{{\rangle}}
\def\brinv{{\langle -1\rangle}}

\def\ID{\mathcal{ID}}

\def\im{{\sf Im}}
\def\re{{\sf Re}}
\def\supp{{\sf supp}}

\begin{document}

\title{Unimodality of the freely selfdecomposable probability laws}

\author{Takahiro Hasebe \and Steen Thorbj{\o}rnsen}

\date{}

\maketitle

\begin{abstract}
We show that any freely selfdecomposable probability law is unimodal.
This is the free probabilistic analog of Yamazato's result in
[Ann.~Probab.~{\bf 6} (1978), 523-531].  

\vspace{2mm}


\vspace{2mm}

\end{abstract}
\section{Introduction}\label{intro}

A.~Ya.~Khintchine introduced the class $\CL$ of limit distributions of certain
independent triangular arrays. It plays an important role in
statistics and mathematical finance, mainly as a consequence of the following
characterization established by P.~L\'evy in 1937: A (Borel-)
probability measure $\mu$ belongs to $\CL$,
if and only if there exists, for any constant $c$ in $(0,1)$,
a probability measure $\mu_c$ on $\R$, such that
\begin{equation}\label{SDa}
\mu=D_c\mu*\mu_c.
\end{equation} 
 Here $D_c\mu$ is the push-forward of $\mu$ by the map $x\mapsto c x$ and $*$ denotes (classical) convolution of probability measures. To distinguish from the corresponding class $\CL(\boxplus)$ in free probability
(described below) we shall
henceforth write $\CL(*)$ instead of just $\CL$. As a result of
L\'evy's characterization the measures in $\CL(*)$ are called
\emph{selfdecomposable}. 
The class $\CL(*)$ contains in particular the
class $\CS(*)$ of stable probability measures on $\R$ as a proper
subclass (see e.g.\ \cite{Sa}).  

A probability measure $\mu$ on $\R$ is called {\it unimodal}, if, for
some $a$ in $\R$, it has the form
\begin{equation}\label{UMa}
\mu(\d t)=\mu(\{a\})\delta_a(\d t)+f(t)\6t,
\end{equation}
where $f\colon\R\to\R$ is increasing (meaning that $x \leq y$ implies $f(x) \leq f(y)$) on $(-\infty,a)$ and decreasing (meaning that $x \leq y$ implies $f(x) \geq f(y)$) on
$(a,\infty)$, and where $\delta_a$ denotes the Dirac measure at $a$.
The problem of unimodality of the measures in $\CL(*)$ emerged
in the 1940's. Already in the original 1949 Russian 
edition of the fundamental book \cite{GnKo} by B.V.~Gnedenko and
A.N.~Kolmogorov it was claimed that all
selfdecomposable distributions are unimodal. However, as explained in
the English translation
\cite{GnKo} (by K.~L.~Chung) there was an error in the proof, and it
took almost 30 
years before a correct proof was obtained by M.~Yamazato in
1978 (see \cite{Ya}). In the appendix to the paper \cite{bp2} from
1999 it was proved by P.~Biane that all measures in the class
$\CS(\boxplus)$ of stable measures with respect to free additive
convolution $\boxplus$ (see Section~\ref{background}) are unimodal. 
In the present paper we
extend this result to the class $\CL(\boxplus)$ of all selfdecomposable
distributions with respect to $\boxplus$; thus establishing a full
free probability analog of Yamazato's result.

In the paper \cite{ht7} it 
was proved by U.~Haagerup and the second named author
that the free analogs of the Gamma distributions (which are
contained in $\CL(\boxplus)\setminus\CS(\boxplus)$) are unimodal, and
the present paper is based in part on techniques from that paper. Let
us also point out that several results from Section~\ref{compact_case}
in the present paper
(most notably Lemma~\ref{P_is_homeomorphism}) may be extracted from
the more general and somewhat differently oriented theory developed
in the papers \cite{Hu1}-\cite{Hu2} by H.-W.~Huang. We prefer in the
present paper to give a completely self-contained and elementary
exposition in the specialized setup considered here. In particular our
approach does not depend upon the rather deep complex analysis
considered in Huang's papers and originating in the work of S.T.~Belinschi
and H.~Bercovici (see e.g.\ \cite{BB05}). 

The remainder of the paper is organized as follows: In
Section~\ref{background} we provide background material on $\boxplus$-infinite
divisibility, the Bercovici-Pata bijection, selfdecomposability and
unimodality. In
Section~\ref{compact_case} we establish unimodality for probability
measures in $\CL(\boxplus)$ satisfying in particular that the
corresponding L\'evy 
measure has a strictly positive $C^2$-density on $\R\setminus\{0\}$. In
Section~\ref{general_case}
we extend the unimodality result from such measures to general measures in
$\CL(\boxplus)$, using that unimodality is preserved under weak
limits.

\section{Background}\label{background}

\subsection{Free and classical infinite divisibility}

A (Borel-) probability measure $\mu$ on
${\mathbb R}$ is called infinitely divisible, if there exists, for each
positive integer $n$, a probability measure $\mu^{1/n}$ on $\R$, such that
\begin{equation}
\mu=\underbrace{\mu^{1/n}*\mu^{1/n}*\cdots*\mu^{1/n}}_{n \ \textrm{terms}},
\label{eqPL.1}
\end{equation}
where $*$ denotes the usual convolution of probability measures (based
on classical independence). We denote by $\ID(*)$ the class of all
such measures on $\R$. We recall that a probability measure $\mu$ on
${\mathbb R}$ is infinitely divisible, if and only if its
characteristic function (or Fourier transform)
$\hat{\mu}$ has the L\'evy-Khintchine representation:
\begin{equation}
\hat{\mu}(u)=\exp\Big[{\rm i}\eta u - {\textstyle\frac{1}{2}}au^2 + 
\int_{{\mathbb R}}\big({\rm e}^{{\rm i}ut}-1-{\rm i}ut 1_{[-1,1]}(t)\big) \
\rho({\rm d}t)\Big], \qquad (u\in{\mathbb R}),
\label{e0.10b}
\end{equation}
where $\eta$ is a real constant, $a$ is a non-negative constant and
$\rho$ is a L\'evy measure on ${\mathbb R}$, meaning that
\[
\rho(\{0\})=0, \quad \textrm{and} \quad \int_{{\mathbb R}}\min\{1,t^2\} \
\rho({\rm d}t)<\infty.
\]
The parameters $a$, $\rho$ and $\eta$
are uniquely determined by $\mu$ and the triplet $(a,\rho,\eta)$ is called the
{\it characteristic triplet} for $\mu$. Alternatively the
L\'evy-Khintchine representation may be written in the form:
\begin{equation}
\hat{\mu}(u)=\exp\Big[{\rm i}\gamma u +
\int_{{\mathbb R}}\Big({\rm e}^{{\rm i}ut}-1-\frac{{\rm
    i}ut}{1+t^2}\Big)\frac{1+t^2}{t^2} \ \sigma({\rm d}t)\Big], 
\quad (u\in{\mathbb R}),
\label{e0.10a}
\end{equation}
where $\gamma$ is a real constant, $\sigma$ is a finite measure on
${\mathbb R}$ and $(\gamma,\sigma)$ is called the \emph{generating
  pair} for $\mu$. 
The relationship between the representations \eqref{e0.10a} and
\eqref{e0.10b} is as follows:
\begin{equation}
\begin{split}
a&=\sigma(\{0\}), \\[.2cm]
\rho({\rm d}t)&=\frac{1+t^2}{t^2}\cdot 1_{{\mathbb R}\setminus\{0\}}(t) \
\sigma({\rm d}t), \\[.2cm]
\eta&=\gamma+\int_{{\mathbb R}}t\Big(1_{[-1,1]}(t)-\frac{1}{1+t^2}\Big) \
\rho({\rm d}t).
\end{split}
\label{ligning3}
\end{equation}

For two probability measures $\mu$ and $\nu$ on $\R$, the free
convolution $\mu\boxplus\nu$ is defined as the spectral distribution
of $x+y$, 
where $x$ and $y$ are \emph{freely independent} (possibly unbounded)
selfadjoint operators on a Hilbert space with spectral distributions
$\mu$ and $\nu$, respectively (see \cite{BV} for further details).
The class $\ID(\boxplus)$ of infinitely divisible probability measures
with respect to free convolution $\boxplus$ is defined by replacing
classical convolution $*$ by free convolution $\boxplus$ in
\eqref{eqPL.1}. 

For a (Borel-) probability measure $\mu$ on $\R$ with support
$\supp(\mu)$, the Cauchy (or Stieltjes) transform is the mapping
$G_\mu\colon\C\setminus\supp(\mu)\to\C$ defined by: 
\begin{equation}
G_\mu(z)=\int_{\R}\frac{1}{z-t}\,\mu(\d t),
\qquad(z\in\C\setminus\supp(\mu)).
\label{eqPL.6}
\end{equation}
The \emph{free cumulant transform} $\CC_\mu$ of $\mu$ is then given by
\begin{equation}
\CC_\mu(z)=zG_\mu^\brinv(z)-1
\label{eqPL.6a}
\end{equation}
for all $z$ in a certain region $R$ of $\C^-$ (the lower half complex
plane), where the (right) inverse 
$G_\mu^\brinv$ of $G_\mu$ is well-defined. Specifically $R$ may be
chosen in the form:
\[
R=\{z\in\C^-\mid \tfrac{1}{z}\in\Delta_{\eta,M}\}, \quad\text{where}\quad
\Delta_{\eta,M}=\{z\in\C^+\mid |\re(z)|<\eta\im(z), \ \im(z)>M\}
\]
for suitable positive numbers $\eta$ and $M$, where $\C^+$ denotes the
upper half complex plane.
It was proved in \cite{BV} (see also
\cite{ma} and \cite{vo2}) that
$\CC_\mu$ constitutes the free analog of $\log\hat{\mu}$ in the sense
that it linearizes free convolution:
\[
\CC_{\mu\boxplus\nu}(z)=\CC_\mu(z)+\CC_{\nu}(z)
\]
for all probability measures $\mu$ and $\nu$ on $\R$ and all $z$ in a
region where all three transforms are defined. 
The results in \cite{BV} are presented in terms of a variant,
$\phi_\mu$, of $\CC_\mu$, which is often referred to as the Voiculescu
transform, and which is again a variant of the $R$-transform $\CR_\mu$
introduced in \cite{vo2}. The relationship is the following: 
\begin{equation}
\phi_{\mu}(z)=\CR_\mu(\tfrac{1}{z})=z\CC_{\mu}(\tfrac{1}{z})
\label{eqPL.6b}
\end{equation}
for all $z$ in a region $\Delta_{\eta,M}$ as above.
In \cite{BV} it was
proved additionally that $\mu\in\ID(\boxplus)$, if and only if
$\phi_\mu$ extends analytically to a map from $\C^+$  into $\C^- \cup
\R$, in which case 
there exists a real constant $\gamma$ and a finite measure $\sigma$ on
$\R$, such that $\phi_\mu$ has the 
\emph{free L\'evy-Khintchine representation}:
\begin{equation}
\label{e1.1a}
\phi_{\mu}(z)=\gamma+\int_{{\mathbb R}}\frac{1+tz}{z-t}\, \sigma({\rm
  d}t),
\qquad (z\in{\mathbb C}^+).
\end{equation}
The pair $(\gamma,\sigma)$ is uniquely determined and is called the
\emph{free generating pair} for $\mu$. 
In terms of the free cumulant transform $\CC_\mu$ the free
L\'evy-Khintchine representation may be written as
\begin{equation}
\mathcal{C}_{\mu}(z) = \eta z+ az^2 + 
\int_{{\mathbb R}}\Big(\frac{1}{1-tz}-1-tz1_{[-1,1]}(t)\Big)
\ \rho({\rm d}t),
\label{eqPL.2}
\end{equation}
where the relationship between the \emph{free characteristic triplet}
$(a,\rho,\eta)$ and the free generating pair $(\gamma,\sigma)$
is again given by \eqref{ligning3}.

In \cite{bp2} Bercovici and Pata introduced a bijection
$\Lambda$ between the two classes $\ID(*)$ and $\ID(\boxplus)$, which
may formally be defined as the mapping sending a measure $\mu$ from
$\ID(*)$ with generating pair $(\gamma,\sigma)$ onto the measure
$\Lambda(\mu)$ in $\ID(\boxplus)$ with \emph{free} generating pair
$(\gamma,\sigma)$. It is then obvious that $\Lambda$ is a bijection, and
it turns out that $\Lambda$ further enjoys the following properties (see
\cite{bp2} and \cite{B-NT02}): 

\begin{enumerate}[a]

\item If $\mu_1,\mu_2\in{\mathcal{ID}}(*)$, then
  $\Lambda(\mu_1*\mu_2)=\Lambda(\mu_1)\boxplus\Lambda(\mu_2)$.

\item If $\mu\in{\mathcal{ID}}(*)$ and $c\in{\mathbb R}$, then
  $\Lambda(D_c\mu)=D_c\Lambda(\mu)$. 

\item For any constant $c$ in ${\mathbb R}$ we have that
  $\Lambda(\delta_c)=\delta_c$, where $\delta_c$ denotes the Dirac
  measure at $c$.

\item $\Lambda$ is a homeomorphism with respect to weak convergence.

\end{enumerate}

The property (d) is equivalent to the free version of Gnedenko's
Theorem: Suppose $\mu,\mu_1,\mu_2,\mu_3,\ldots$ is a sequence of measures
from $\ID(\boxplus)$ with \emph{free} generating pairs:
$(\gamma,\sigma),(\gamma_1,\sigma_1),(\gamma_2,\sigma_2),(\gamma_3,\sigma_3),
\ldots$, respectively. Then 
\begin{equation}
\mu_n\overset{\rm w}{\longrightarrow}\mu \iff
\gamma_n\longrightarrow\gamma \ \text{and} \ \sigma_n\overset{\rm
w}{\longrightarrow}\sigma. 
\label{free_Gnedenko_thm}
\end{equation}
(cf.\ Theorem~3.8 in \cite{B-NT02}) 

\subsection{Selfdecomposability and Unimodality}
The selfdecomposablity defined in \eqref{SDa} has an equivalent characterization: a probability measure $\mu$ is in $\CL(*)$ if and only if $\mu$ is in $\ID(*)$ and the L\'evy measure (cf.\ \eqref{e0.10b})
has the form 
\begin{equation}\label{spectral}
\rho(\d t)=\frac{k(t)}{|t|}\6t, 
\end{equation}
where $k\colon\R\setminus\{0\}\to[0,\infty)$ is increasing on $(-\infty,0)$ and decreasing on $(0,\infty)$ (see \cite{Sa}).

In analogy with the class $\CL(*)$,
a probability measure $\mu$ on $\R$ is called
$\boxplus$-\emph{selfdecomposable}, if there exists, for any $c$ in
$(0,1)$, a probability measure $\mu_c$ on $\R$, such that
\begin{equation}
\mu=D_c\mu\boxplus\mu_c. 
\end{equation}
Denoting by $\CL(\boxplus)$ the class of
such measures, it follows from the properties of $\Lambda$ that
\begin{equation}
 \Lambda(\CL(*))=\CL(\boxplus)
\label{eq9}
\end{equation}
(see \cite{B-NT02}). By the definition of $\Lambda$ and \eqref{eq9}, if we let the
term ``L\'evy measure'' refer to the free L\'evy-Khintchine 
representation \eqref{eqPL.2} rather than the classical one
\eqref{e0.10b}, 
 then we have exactly the same characterization of the measures in $\CL(\boxplus)$: a probability measure $\mu$ is in $\CL(\boxplus)$ if and only if its L\'evy measure in \eqref{eqPL.2} is of the form \eqref{spectral}. 
 
The definition of a unimodal probability measure $\mu$ given in
Section~\ref{intro} is equivalent to the existence of a real number
$a$, such that the distribution function $t\mapsto\mu((-\infty,t])$ is 
convex (i.e.\ $\mu((-\infty, p s+ q t]) \leq p \mu((-\infty,s])+q\mu((-\infty,t])$ for all $s,t$ and all $ p,q\geq 0, p+q=1$) on $(-\infty,a)$ and concave on $(a,\infty)$. From this
characterization it follows that for any sequence
$\mu,\mu_1,\mu_2,\mu_3,\ldots$ of 
probability measures on $\R$ we have the implication:
\begin{equation}
\text{$\mu_n$ is unimodal for all $n$ and $\mu_n\overset{\rm
    w}{\longrightarrow}\mu$}  \ \Longrightarrow \
\text{$\mu$ is unimodal}
\label{unimodal_vs_weak_conv}
\end{equation}
(see e.g.\ \cite[\S32, Theorem~4]{GnKo}).

\subsection{Lindel\"of's Theorem}

In this subsection we present a variant (Lemma~\ref{Lindeloef} below)
of Lindel\"of's Theorem (see \cite{Li} or \cite[Theorem~2.2]{CoLo}),
which plays a crucial role in
Section~\ref{compact_case} in combination with Stieltjes inversion.
Before stating the lemma we introduce some notation:
For any number $\delta$ in $(0,\pi)$ we put
\[
\bigtriangledown_\delta=\big\{r\e^{\ri\theta}\bigm|
\delta<\theta<\pi-\delta, \ r>0\big\}.
\]

\begin{lemma}\label{Lindeloef}
Let $G\colon\C^+\to\C^-$ be an analytic function, and assume that
there exists a curve $(z_t)_{t\in[0,1)}$ in $\C^+$, such that
$\lim_{t\to1}z_t=0$, and such that $\alpha:=\lim_{t\to1}G(z_t)$ exists in $\C$.
Then for any number $\delta$ in $(0,\pi)$ we also have that
$
\lim_{z\to0, z\in\bigtriangledown_\delta}G(z)=\alpha,
$
i.e., $G$ has non-tangential limit $\alpha$ at $0$.
\end{lemma}

Lemma~\ref{Lindeloef} may e.g.\ be derived from
Theorem~2.2 in \cite{CoLo}, which provides a similar result
for (in particular) \emph{bounded} analytic functions 
$f\colon\{x+\ri y\mid x>0, \ y\in\R\}\to\C$. Recalling that the
mapping $\zeta\mapsto\frac{\zeta-1}{\zeta+1}$ is a conformal bijection
of $\{x+\ri y\mid x>0, \ y\in\R\}$ onto the open unit disc in $\C$,
Lemma~2.1 then follows by applying \cite[Theorem~2.2]{CoLo} to
the bounded function
\[
f(z)=\frac{\ri G(\e^{\ri\frac{\pi}{2}}z)-1}{\ri
  G(\e^{\ri\frac{\pi}{2}}z)+1}, \qquad
(z\in\{x+\ri y\mid x>0, \ y\in\R\}).
\]

\section{The case of L\'evy measures with positive density on
  $\R$}\label{compact_case}

In this section we prove unimodality for measures in $\CL(\boxplus)$
with L\'evy measures in the form $\frac{k(t)}{|t|}$, where $k$ 
satisfies the conditions (a)-(c) listed below. In a previous version
of the manuscript we considered the case where $k$ is
compactly supported, but in that setting some proofs become more
delicate and complicated than the ones to follow.

Throughout the remaining part of this section we consider a function
$k\colon\R\setminus\{0\}\to[0,\infty)$ such that  
\begin{enumerate}[\rm a]
\item $k$ is $C^2$ and $(1+t^2)^m k^{(n)}(t)$ are bounded for $m, n\in\{0,1,2\}$,  
\item $k$ is increasing on $(-\infty,0)$, decreasing on $(0,\infty)$, 
\item $k$ is strictly positive on $\R\setminus\{0\}$. 
\end{enumerate}


Next we define
\begin{equation*}
\begin{split}
\tilde{k}(t)&={\rm sign}(t)k(t), \qquad(t\in\R),
\\[.2cm] 
\tilde{G}_k(z)&=\int_{\R}\frac{\tilde{k}(t)}{z-t}\6t,
\qquad(z\in\C^+),
\\[.2cm]
H_k(z)&=z+z\tilde{G}_k(z), \qquad(z\in\C^+).
\end{split}
\end{equation*}
We note for later use that
\begin{equation}
H_k(z)=z+z\int_{\R}\frac{\tilde{k}(t)}{z-t}\6t
=z+\int_\R\Big(1+\frac{t}{z-t}\Big)\tilde{k}(t)\6t
=z+\gamma_k+\int_\R\frac{|t|k(t)}{z-t}\6t,
\label{eq5}
\end{equation}
where we have introduced $\gamma_k=\int_\R\tilde{k}(t)\6t$.

In the following we shall consider additionally the auxiliary function
$F_k\colon\C^+\to(0,\infty)$ given by
\begin{equation}
F_k(x+\ri y)=\int_\R\frac{|t|k(t)}{(x-t)^2+y^2}\6t, \qquad(x+\ri y\in\C^+), 
\label{eq2}
\end{equation}
which satisfies $F_k(z) \im(z) = \im(z- H_k(z))$. 

\begin{lemma}\label{intro_v}
\begin{enumerate}

\item[\rm (i)]
For all $x$ in $\R$ there exists a unique number $y=v_k(x)$ in
$(0,\infty)$ such that
\begin{equation}
F_k(x+\ri v_k(x))=\int_\R\frac{|t|k(t)}{(x-t)^2+v_k(x)^2}\6t=1.
\label{eq3}
\end{equation}

\item[\rm (ii)] We have that
\[
\CG:=\{z\in\C^+\mid H_k(z)\in\R\}=\{x+\ri v_k(x)\mid x\in\R\}.
\]

\item[\rm (iii)]
We have that
\[
\CG^+:=\{z\in\C^+\mid H_k(z)\in\C^+\}=\{x+\ri y\mid x\in\R, \ y>v_k(x)\}.
\]

\item[\rm (iv)] The function $v_k\colon\R\to(0,\infty)$ is analytic on $\R$.
  
  \item[\rm(v)] We have that 
  $$
  \lim_{|x|\to\infty}v_k(x)=0. 
  $$


\end{enumerate}
\end{lemma}

\begin{proof}

(i) \ For any $x$ in $\R$ the function
\[
y\mapsto\int_\R\frac{|t|k(t)}{(x-t)^2+y^2}\6t, \qquad(y\in(0,\infty))
\]
takes values in $(0,\infty)$ and is continuous (by dominated
convergence) and strictly decreasing in $y$. Since $k$ is strictly
positive and continuous we find additionally that
\[
\lim_{y\searrow0}\int_\R\frac{|t|k(t)}{(x-t)^2+y^2}\6t
=\infty, \qand \lim_{y\nearrow\infty}\int_\R\frac{|t|k(t)}{(x-t)^2+y^2}\6t=0
\]
by monotone and dominated convergence. Hence there is a
unique $y=v_k(x)$ in $(0,\infty)$ such that 
$\int_\R\frac{|t|k(t)}{(x-t)^2+y^2}\6t=1$.

(ii) \ For any $x,y$ in $\R$, such that $y>0$, we note that
\begin{equation}
\begin{split}
\im\big(H_k(x+\ri y)\big)
&=y+\im\Big(\int_\R\frac{x+\ri y}{x+\ri y-t}\tilde{k}(t)\6t\Big)
=y+\int_\R\frac{y(x-t)-yx}{(x-t)^2+y^2}\tilde{k}(t)\6t
\\[.2cm]
&=y\Big(1-\int_\R\frac{t\tilde{k}(t)}{(x-t)^2+y^2}\6t\Big)
=y\Big(1-\int_\R\frac{|t|k(t)}{(x-t)^2+y^2}\6t\Big).
\end{split}
\label{eq0}
\end{equation}
Hence it follows that
\begin{equation}
\im\big(H_k(x+\ri y)\big)=0
\iff \int_\R\frac{|t|k(t)}{(x-t)^2+y^2}\6t=1.
\label{eq1}
\end{equation}
The right hand side of \eqref{eq1} holds, if and
only if $y=v_k(x)$. 

(iii) \ It is apparent that
$\int_\R\frac{|t|k(t)}{(x-t)^2+y^2}\6t<1$ for any $x$ in $\R$ and  
all $y$ in $(v_k(x),\infty)$.  In combination with \eqref{eq0} this shows that
$\CG^+=\{x+\ri y\mid x\in\R, \ y>v_k(x)\}$ as desired.

(iv) \ Consider the function
$\tilde{F}_k\colon\R\times(0,\infty)\to\R$ given by
\begin{equation*}
\tilde{F}_k(x,y)=F_k(x+\ri y)=\int_{\R}\frac{|t|k(t)}{(x-t)^2+y^2}
=1-y^{-1}\im\big(H_k(x+\ri y)), \quad((x,y)\in\R\times(0,\infty)).
\end{equation*}
Since $H_k$ is analytic on $\C^+$ it follows that $\tilde{F}_k$ is
analytic on $\R\times(0,\infty)$. By differentiation under the
integral sign we note in particular that
\[
\frac{\partial}{\partial y}\tilde{F}_k(x,y)
=-2y\int_\R\frac{|t|k(t)}{((x-t)^2+y^2)^2}\6t<0
\]
for all $(x,y)$ in $\R\times(0,\infty)$. Since $v_k(x)>0$ and
$\tilde{F}_k(x+\ri v_k(x))=1$ for all $x$ in $\R$ it follows then from
the Implicit Function Theorem (for analytic functions; see
\cite[Theorem~7.6]{FG}) that $v_k$ is analytic on $\R$.

(v) \ By dominated convergence $\lim_{|x|\to\infty} F_k(x+\ri y)=0$
for any fixed $y$ in $(0,\infty)$. Hence (v) follows from \eqref{eq3}
and the fact that $y \mapsto F_k(x+\ri y)$ is decreasing (for fixed $x$). 
\end{proof}

\begin{lemma}\label{formel_for_Cauchytransform}
Let $\nu_k$ be the measure in $\ID(\boxplus)$ with free characteristic
triplet $(0,\frac{k(t)}{|t|}\6t,\int_{-1}^1\tilde{k}(t)\6t)$.
Then the Cauchy transform $G_{\nu_k}$ of $\nu_k$ satisfies the identity:
\[
G_{\nu_k}(H_k(z))=\frac{1}{z}
\]
for all $z$ in $\C^+$ such that $H_k(z)\in\C^+$.
\end{lemma}

\begin{proof} 

Let $\CC_{\nu_k}$ denote the free cumulant transform of $\nu_k$
(extended to all of $\C^-$).
For any $w$ in $\C^-$ we then find (cf.\ formula \eqref{eqPL.2}) that 
\begin{equation*}
\begin{split}
\CC_{\nu_k}(w)&=w\int_{-1}^1\tilde{k}(t)\6t
+\int_\R\Big(\frac{1}{1-wt}-1-wt1_{[-1,1]}(t)\Big)\frac{k(t)}{|t|}\6t 
\\[.2cm]
&=\int_\R\Big(\frac{1}{1-wt}-1\Big)\frac{k(t)}{|t|}\6t
\\[.2cm]
&=w\int_\R\frac{t}{1-wt}\frac{k(t)}{|t|}\6t
=w\int_\R\frac{\tilde{k}(t)}{1-wt}\6t.
\end{split}
\end{equation*}
Setting $w=\frac{1}{z}$ it follows for any $z$ in $\C^+$ that
\[
\CC_{\nu_k}\big(\tfrac{1}{z}\big)=\frac{1}{z}
\int_\R\frac{\tilde{k}(t)}{1-\frac{t}{z}}\6t
=\int_\R\frac{\tilde{k}(t)}{z-t}\6t
=\tilde{G}_k(z).
\]
By definition of the free cumulant transform it therefore follows that
\[
\frac{1}{z}G_{\nu_k}^\brinv(\tfrac{1}{z})-1=\CC_{\nu_k}(\tfrac{1}{z})
=\tilde{G}_k(z), 
\]
and hence that
\[
G_{\nu_k}^\brinv(\tfrac{1}{z})=z\tilde{G}_k(z)+z=H_k(z)
\]
for all $z$ in a suitable region $\Delta_{\eta,M}$, 
where $\eta,M>0$. We may thus conclude that
\begin{equation}
\frac{1}{z}=G_{\nu_k}(H_k(z))
\label{eq1f}
\end{equation}
for all $z$ in $\Delta_{\eta,M}$, but since $\{z\in\C^+\mid
H_k(z)\in\C^+\}$ is a connected region of $\C^+$ (cf.\
Lemma~\ref{intro_v}(iii)), the identity \eqref{eq1f} extends to all
$z$ in this region by analytic continuation.
\end{proof}

In the following we consider the function $P_k\colon\R\to\R$ defined by
\begin{equation}
P_k(x)=H_k(x+\ri v_k(x)), \qquad(x\in\R).
\label{defP}
\end{equation}

\begin{proposition}\label{non-tangential_limit}
For any $x$ in $\R$ we have that
\[
G_{\nu_k}(z)\longrightarrow \frac{1}{x+\ri v_k(x)} \qquad\text{as $z\to
  P_k(x)$ non-tangentially from $\C^+$.}
\]
\end{proposition}

\begin{proof}  For any $s$ 
in $[0,1]$ we put $w_s=x+\ri(v_k(x)+s)$, so that $w_s\in\CG^+$ for all
$s$ in $(0,1]$ according to Lemma~\ref{intro_v}(iii). Moreover, since
$H_k$ is analytic on $\C^+$, and 
$w_s\in\C^+$ for all $s$ in $[0,1]$, it follows that
\[
H_k(w_s)\longrightarrow H_k(w_0)=H_k(x+\ri v_k(x))=P_k(x)\in\R \qquad\text{as
  $s\searrow0$}.
\]
In addition it follows from Lemma~\ref{formel_for_Cauchytransform}
that
\[
G_{\nu_k}(H_k(w_s))=\frac{1}{w_s}=\frac{1}{x+\ri(v_k(x)+s)}
\longrightarrow\frac{1}{x+\ri v_k(x)} \qquad\text{as $s\searrow0$}. 
\]
Thus $G_{\nu_k}(z)$ has the limit $\frac{1}{x+\ri v_k(x)}$ as $z\to P_k(x)$
along the curve $s\mapsto H_k(w_s)$. It follows then from
Lemma~\ref{Lindeloef} that in
fact $G_{\nu_k}(z)\to\frac{1}{x+\ri v_k(x)}$ as $z\to P_k(x)$ non-tangentially
from $\C^+$, as desired.
\end{proof}

\begin{lemma}\label{P_is_homeomorphism}
The function $P_k$ is a strictly increasing homeomorphism of $\R$ onto $\R$.
\end{lemma}

\begin{proof} We show first that $P_k(x)\to\pm\infty$ as
  $x\to\pm\infty$. From Lemma~\ref{intro_v}(v), formula~\eqref{defP}
  and  formula~\eqref{eq5} this will follow, if we show that 
\begin{equation}\label{eqB}
\sup_{y \in (0,1/2)}\Big|\int_\R\frac{|t|k(t)}{x+\ri y-t}\6t\Big|\longrightarrow0
\text{~as~} |x|\to \infty.
\end{equation}
Consider in the following $x$ in $\R\setminus[-2,2]$ and $y,\delta$ in
$(0,\frac{1}{2})$.
We then divide the integral as follows:
\begin{align}\label{eqA}
\int_\R\frac{|t|k(t)}{x+\ri y-t}\6t 
&= \int_{x-\delta}^{x+\delta}\frac{|t|k(t)}{x+\ri y-t}\6t +
\int_{\R\setminus[x-\delta,x+\delta]}\frac{|t|k(t)}{x+\ri y-t}\6t.
\end{align}
To estimate the first term on the right hand side of \eqref{eqA}, we
perform integration by parts:  
 \begin{align*}
 \int_{x-\delta}^{x+\delta}\frac{|t|k(t)}{x+\ri y-t}\6t
 &= \Big[-\log(x-t+\ri y)|t|k(t)\Big]_{x-\delta}^{x+\delta} 
+\int_{x-\delta}^{x+\delta}\log(x-t+\ri y)\frac{\d}{\d t}\big(|t|k(t)\big)\6t,  
\end{align*}
where $\log$ is the principal branch, i.e.,
\[
\log(x-t+\ri y)=\tfrac{1}{2}\log((x-t)^2+y^2)+\ri{\rm Arg}(x-t+\ri y),
\]
where ${\rm Arg}$ is the principal argument. Given any positive number
$\epsilon$, we choose next $\delta$ in $(0,1/2)$ such that
\[
\int_{-\delta}^{\delta}
\sqrt{\pi^2+(\log|t|)^2}\6t \leq \epsilon\Big(\sup_{|t| \geq
  1}\Big|\frac{\d}{\d t}\big(|t| k(t)\big)\Big|\Big)^{-1}. 
\]
Since $t\mapsto|\log(t)|$ is decreasing on $(0,1)$, it follows then that
\begin{align*}
 \Big|\int_{x-\delta}^{x+\delta}\log(x-t+\ri
   y)\frac{\d}{\d t}\big(|t|k(t)\big)\6t\Big| 
 &\leq\int_{x-\delta}^{x+\delta}\sqrt{\pi^2+(\log|x-t|)^2}
 \Big|\frac{\6}{\6t}\big(|t|k(t)\big)\Big|\6t \\ 
 &\leq \sup_{|t| \geq 1} \Big|\frac{\d}{\d t}\big(|t| k(t)\big)\Big|
 \int_{-\delta}^{\delta}\sqrt{\pi^2+(\log|t|)^2}\6t \leq \epsilon.
\end{align*}
Since $|t|k(t)\to0$ as
$|t|\to\infty$ (cf.\ condition (a) above), we note further that
$$
\Big|\Big[-\log(x-t+\ri y)|t|k(t)\Big]_{x-\delta}^{x+\delta}\Big| 
\leq 2\cdot \sqrt{\pi^2+(\log \delta)^2} \cdot
\max\big\{|x-\delta|k(x-\delta),|x+\delta|k(x+\delta)\big\}
\leq \epsilon,
$$
for any $y$ in $(0,1/2)$ and all $x$ with $|x|$ sufficiently large. Thus the
first term of \eqref{eqA} is 
bounded by $2\epsilon$ whenever $|x|$ is large enough, uniformly in $y
\in(0,1/2)$.

Regarding the second term on the right hand side of \eqref{eqA} we
note first that 
$\lim_{|x|\to\infty}\int_{\R\setminus[x-\delta,x+\delta]}\frac{|t|k(t)}{|x-t|}\6t=0$
by dominated convergence. Therefore
\begin{align*} 
\sup_{y\in\R}
\Big|\int_{\R\setminus[x-\delta,x+\delta]}\frac{|t|k(t)}{x+\ri y-t}\6t\Big| 
\leq\int_{\R\setminus[x-\delta,x+\delta]}\frac{|t|k(t)}{|x-t|}\6t 
\leq \epsilon,
\end{align*}
whenever $|x|$ is sufficiently large. Thus we have established
\eqref{eqB}. 

It remains now to show that $P_k$ is injective and continuous on
$\R$, since these properties are then automatically transferred to the
inverse $P_k^\brinv$. The continuity is obvious from the continuity of
$v_k$ (cf.\ formula~\ref{defP}).
To see that $P_k$ is injective on $\R$, assume that
$x,x'\in\R$ such that $P_k(x)=P_k(x')$. Then
Proposition~\ref{non-tangential_limit} shows that 
\[
\frac{1}{x+\ri v_k(x)}=\lim_{z\overset{\sphericalangle}{\to}P_k(x)}G_{\nu_k}(z)
=\lim_{z\overset{\sphericalangle}{\to}P_k(x')}G_{\nu_k}(z)=\frac{1}{x'+\ri v_k(x')},
\]
where ``$\overset{\sphericalangle}{\to}$'' denotes non-tangential
limits. Clearly the above identities imply that $x=x'$.
\end{proof}

\begin{corollary}\label{density}
The measure $\nu_k$ is absolutely continuous with respect to Lebesgue
measure with a continuous density $f_{\nu_k}$ given by
\[
f_{\nu_k}(P_k(x))=\frac{v_k(x)}{\pi(x^2+v_k(x)^2)}, \qquad(x\in\R).
\]
\end{corollary}
In particular, the support of $\nu_k$ is $\R$. 
\begin{proof}
This follows by Stieltjes-Inversion and
Proposition~\ref{non-tangential_limit}. Indeed, for any $x$ in $\R$ we
have that
\[
\lim_{y\searrow0}G_{\nu_k}(P_k(x)+\ri y)=\frac{1}{x+\ri v_k(x)}.
\]
Recalling (see e.g.\ Chapter~XIII in \cite{RS}) that the singular part
of $\nu_k$ is concentrated on the set 
\[
\big\{\xi\in\R\bigm| 
\textstyle{\lim_{y\searrow0}|G_{\nu_k}(\xi+\ri y)|=\infty}\big\},
\]
it follows in particular that
$\nu_k$ has no singular part. For any $x$ in $\R$ we 
find furthermore by the Stieltjes Inversion Formula that
\[
f_{\nu_k}(P_k(x))=\frac{-1}{\pi}\lim_{y\searrow0}\im(G_{\nu_k}(P_k(x)+\ri y))
=\frac{-1}{\pi}\im\Big(\frac{1}{x+\ri v_k(x)}\Big)
=\frac{v_k(x)}{\pi(x^2+v_k(x)^2)}.
\]
In particular we see that $f_{\nu_k}(\xi)>0$ for any $\xi$ in $\R$. Denoting by $P_k^\brinv$ the inverse of $P_k$,
we note finally that
\[
f_{\nu_k}(\xi)
=\frac{v_k(P_k^\brinv(\xi))}{\pi(P_k^{\brinv}(\xi)^2+v_k(P_k^\brinv(\xi))^2)} 
\qquad(\xi\in\R),
\]
which via the continuity of $P_k^\brinv$ and $v_k$ shows that
$f_{\nu_k}$ is continuous too. 
\end{proof}

\begin{remark}
Corollary \ref{density} is a special case of Huang's density formula
for freely infinitely divisible distributions \cite[Theorem
3.10]{Hu2}, which does not impose any assumptions on the L\'evy
measure. Our approach is similar to that of Biane in \cite{B97}. For
example his function $\psi_t$ resembles our function $P_k$.  
\end{remark}  

The next lemma is key to the main result on unimodality.  

\begin{lemma}\label{key-lemma}
Consider the function $F_k$ defined by \eqref{eq2}.
Then for any $r$ in $(0,\infty)$ there exists a number $\theta_r$ in
$(0,\pi)$ such that the function 
$$
\theta\mapsto
F_k(r\sin(\theta)\e^{\ri\theta})
$$ 
is strictly decreasing on $(0,\theta_r]$ and
strictly increasing on $[\theta_r,\pi)$.
\end{lemma}

\begin{proof}
We introduce a new variable $u$ by setting $t=(r\sin\theta)u$. Then  
$$
F_k(r\sin(\theta)\e^{\ri\theta})= \int_{\R}\frac{|u| k(ru \sin\theta)}{1-2u\cos\theta +u^2}\6u,\qquad (\theta\in(0,\pi)). 
$$
Now consider any decreasing function $h\colon(0,\infty)\to(0,\infty)$ from
$C^2((0,\infty))$ satisfying that the functions $(1+t^2)^m h^{(n)}(t)$
are bounded for any
$m,n$ in $\{0,1,2\}$. These assumptions ensure in particular that we may
define $\psi_h\colon(-1,1)\to\R$ by
$$
\psi_h(x):=\int_0^\infty \frac{u}{1-2xu+u^2}h(u\sqrt{1-x^2})\6u,\qquad
(x\in(-1,1)).  
$$
Note then that if we define $k_r^\pm(u):=k(\pm ru)$ for $u$ in $(0,\infty)$, and
\begin{equation}
\Psi_r(x)=\psi_{k_r^+}(x)+\psi_{k_r^-}(-x), \qquad(x\in(-1,1)),
\label{eq_def_Phi}
\end{equation}
then it holds that
\begin{equation}\label{eq+-}
F_k(r\sin(\theta)\e^{\ri\theta})=\Psi_r(\cos\theta), \qquad(\theta\in(0,\pi)).
\end{equation}
We show in the following that
\begin{enumerate}[1]
\item\label{1} $\psi_h'(x)>0$ for $x$ in $(0,1)$,
\item\label{2} $\psi_h'(x)<0$ for $x$ in $(-1,-\frac{\sqrt{2}}{2}]$, 
\item\label{3} $\psi_h''(x)>0$ for $x$ in
  $[-\frac{\sqrt{3}}{2},\frac{\sqrt{3}}{2}]$. 
\end{enumerate}
Before establishing these conditions we remark that the assumptions on
$h$ ensure, that we may perform differentiation under the integral sign
and integration by parts as needed in the following, and we shall do
so without further notice.

For any $x$ in $(-1,1)$ we note first by differentiation under the
integral sign that
\begin{equation}
\begin{split}
\psi_h'(x)
&= \int_0^\infty \frac{2u^2}{(1-2ux+u^2)^2}\,h(u\sqrt{1-x^2})\6u \\
&~~~-\int_0^\infty \frac{u^2}{1-2ux+u^2}\cdot\frac{x}{\sqrt{1-x^2}}\,
h'(u\sqrt{1-x^2})\6u,  
\end{split}
\end{equation}
which shows that \ref{1} holds. Moreover, integration by parts yields that  
\begin{equation}
\begin{split}
\psi_h'(x)
&= \int_0^\infty \frac{2u^2}{(1-2ux+u^2)^2}\,h(u\sqrt{1-x^2})\6u
\\
&~~~+\int_0^\infty \frac{\partial}{\partial u}
\left(\frac{u^2}{1-2ux+u^2}\right)\cdot\frac{x}{1-x^2}\,h(u\sqrt{1-x^2})\6u
\\
&=\int_0^\infty\frac{2u((1-2x^2)u+x)}{(1-2xu+u^2)^2(1-x^2)}\,h(u\sqrt{1-x^2})\6u,
\end{split}
\end{equation}
which verifies \ref{2}. 

Finally, we proceed to compute $\psi_h''(x)$. Using Leibniz' formula we
find that
\begin{equation}
\begin{split}
\psi_h''(x)
&= \int_0^\infty \frac{8u^3}{(1-2ux+u^2)^3}\,h(u\sqrt{1-x^2})\6u \\ 
&~~~~-\int_0^\infty \frac{4u^3}{(1-2ux+u^2)^2}\cdot\frac{x}{\sqrt{1-x^2}}\,h'(u\sqrt{1-x^2})\6u \\ 
&~~~~  -\int_0^\infty \frac{u^2}{1-2ux+u^2}\left(\frac{1}{\sqrt{1-x^2}} +\frac{x^2}{(1-x^2)^{3/2}}\right)\,h'(u\sqrt{1-x^2})\6u\\
&~~~~ +  \int_0^\infty \frac{u^3}{1-2ux+u^2}\cdot\frac{x^2}{1-x^2}\,h''(u\sqrt{1-x^2})\6u \\
&= \int_0^\infty \frac{8u^3}{(1-2ux+u^2)^3}\,h(u\sqrt{1-x^2})\6u \\ 
&~~~~-\int_0^\infty \frac{u^2(1+2ux+u^2)}{(1-2ux+u^2)^2\sqrt{1-x^2}}\,h'(u\sqrt{1-x^2})\6u \\ 
&~~~~  -\int_0^\infty \frac{u^2}{1-2ux+u^2}\cdot\frac{x^2}{(1-x^2)^{3/2}}\,h'(u\sqrt{1-x^2})\6u\\
&~~~~ +  \int_0^\infty \frac{u^3}{1-2ux+u^2}\cdot\frac{x^2}{1-x^2}\,h''(u\sqrt{1-x^2})\6u.  
\end{split}
\end{equation}
In the resulting expression above the first three integrals are
positive for any $x$ in $(-1,1)$, since $-h',h\ge0$ and $u^2+2ux+1 = (u+x)^2+1-x^2 \geq 0$.
By integration by parts, the last integral can be re-written as follows: 
\begin{equation}\label{parts}
\begin{split}
& \int_0^\infty \frac{u^3}{1-2ux+u^2}
\cdot\frac{x^2}{1-x^2}\,h''(u\sqrt{1-x^2})\6u \\
 &~~~~=-\frac{x^2}{(1-x^2)^{3/2}}\int_0^\infty\frac{\partial}{\partial
   u}\left( \frac{u^3}{1-2xu+u^2}\right) \cdot h'(u\sqrt{1-x^2})\6u\\ 
 &~~~~=   -\frac{x^2}{(1-x^2)^{3/2}}\int_0^\infty
 \frac{u^2}{(1-2xu+u^2)^2}
\left((u-2x)^2+3-4x^2\right) h'(u\sqrt{1-x^2})\6u.
\end{split}
\end{equation}
Hence this integral is positive as well for any $x$ in
$[-\frac{\sqrt{3}}{2},\frac{\sqrt{3}}{2}]$, and altogether the property
\ref{3} is established.

Recalling now formula \eqref{eq_def_Phi}, note that 
it follows from conditions
\ref{1}-\ref{3} that
$\Psi_r'(x) = \psi_{k_r^+}'(x)-\psi_{k_r^-}'(-x) >0$, if $x\geq
\frac{\sqrt{2}}{2}$, $\Psi_r'(x) <0$, if $x\leq- \frac{\sqrt{2}}{2}$,
and $\Psi_r''(x) = \psi_{k_r^+}''(x)+\psi_{k_r^-}''(-x) >0$, if $|x|\leq
\frac{\sqrt{3}}{2}$.   
Hence, $\Psi_r '$ is strictly increasing on 
$(-\frac{\sqrt{3}}{2}, \frac{\sqrt{3}}{2})$ and there exists a unique
zero of $\Psi_r'$ at some $x_r$ in $(-\frac{\sqrt{2}}{2},
\frac{\sqrt{2}}{2})$. Therefore $\Psi_r$ is strictly decreasing on
$(-1,x_r]$ and strictly increasing on $[x_r,1)$, and the lemma now
follows readily from formula \eqref{eq+-}.
\end{proof}

\begin{proposition}\label{unimodality_I}
Consider a function $k\colon\R\setminus\{0\}\to[0,\infty)$ which
satisfies conditions (a)-(c) listed in the beginning of this section.
Then the associated measure $\nu_k$ (described
in Lemma~\ref{formel_for_Cauchytransform}) is unimodal. 
In fact there exists a number $\omega$ in $\R$,
such that the density $f_{\nu_k}$ (cf.\ Corollary~\ref{density})
is strictly increasing on $(-\infty,\omega]$ and
strictly decreasing on $[\omega,\infty)$.
\end{proposition}

\begin{proof} We show first for any number $\rho$ in $(0,\infty)$
 that the equality $f_{\nu_k}(\xi)=\rho$ has at most two solutions in
  $\xi$. Since $P_k$ is a bijection of $\R$ onto itself, this is
  equivalent to showing that the equality
\[
\rho=f_{\nu_k}(P_k(x))=\frac{v_k(x)}{\pi(x^2+v_k(x)^2)}
\]
has at most two solutions in $x$. For this we note first that
\[
\big\{x+\ri y\in\C^+\bigm| \tfrac{y}{\pi(x^2+y^2)}=\rho\big\}
=C_\rho\setminus\{0\},
\]
where $C_\rho$ is the circle in $\C$ with center
$\frac{\ri}{2\pi\rho}$ and radius 
$\frac{1}{2\pi\rho}$. Writing $x+\ri y$ as
$r\e^{\ri\theta}$ ($r>0$, $\theta\in(-\pi,\pi]$) we find that $C_\rho$ is given by
\[
C_\rho=\big\{\tfrac{1}{\pi\rho}\sin(\theta)\e^{\ri\theta}\bigm|
\theta\in(0,\pi]\big\}
\]
in polar coordinates. We need to show that $C_\rho$ intersects the
graph $\CG$ of $v_k$ in 
at most two points. By the defining property \eqref{eq3} of $v_k$, this
is equivalent to showing that the equality
\[
F_k\big(\tfrac{1}{\pi\rho}\sin(\theta)\e^{\ri\theta}\big)=1
\]
has at most two solutions for $\theta$ in $(0,\pi)$. But this follows
immediately from Lemma~\ref{key-lemma}.

It is now elementary to check that $\nu_k$ is unimodal. Since
$f_{\nu_k}$ is continuous and strictly positive on $\R$, and since
$f_{\nu_k}(x)\to0$ as $x\to\pm\infty$ (cf.\ Corollary~\ref{density}),
$f_{\nu_k}$ attains a strictly positive global maximum at
some point $\omega$ in $\R$. If $f_{\nu_k}$ was not
increasing on $(-\infty,\omega]$, then we could choose $\xi_1,\xi_2$ in
$(-\infty,\omega)$ such that $\xi_1<\xi_2$, and
$f_{\nu_k}(\xi_1)>f_{\nu_k}(\xi_2)>0$. Choosing any number $\rho$ in
$(f(\xi_2),f(\xi_1))$, it follows then from the continuity of
$f_{\nu_k}$, that each of the intervals $(-\infty,\xi_1)$, $(\xi_1,\xi_2)$
and $(\xi_2,\omega)$ must contain a solution to the equation
$f_{\nu_k}(\xi)=\rho$, which contradicts what we established above.
Subsequently the argumentation given above also implies
that $f_{\nu_k}$ is in fact \emph{strictly} increasing on
$(-\infty,\omega]$. Similarly it follows that $f_{\nu_k}$ must be
strictly decreasing on $[\omega,\infty)$, and this completes the proof.
\end{proof}

\section{The general case}\label{general_case}

In this section we extend Proposition~\ref{unimodality_I} to general
measures $\nu$ from $\CL(\boxplus)$. The key step is the following
approximation result. 

\begin{lemma}\label{approximation_lemma}
Let $k\colon\R\setminus\{0\}\to[0,\infty)$ be a
  function as in \eqref{spectral} such that $\frac{k(t)}{|t|}1_{\R\setminus\{0\}}(t)\6t$ is a L\'evy
  measure. Let further $a$ be a non-negative number.

  Then there exists a sequence $(k_n)$ of functions
  $k_n\colon\R\setminus\{0\}\to[0,\infty)$,
  satisfying the conditions (a)-(c) in Section \ref{compact_case}, such that
\[
\frac{|t|k_n(t)}{1+t^2}\6t\overset{\rm w}{\longrightarrow}
a\delta_0+\frac{|t|k(t)}{1+t^2}\6t
\]
as $n\to\infty$.
\end{lemma}

\begin{proof}
For each $n$ in $\N$ we introduce first the function
$k_n^0\colon\R\to[0,\infty)$ defined by
\[
k_n^0(t)=
\begin{cases}
0,& \text{if~}t\in(-\infty,0],\\
k(\tfrac{1}{n}), &\text{if $t\in(0,\frac{1}{n})$,}\\
k(t), &\text{if $t\in[\frac{1}{n},n]$,}\\
0, &\text{if $t\in(n,\infty)$},
\end{cases}
\]
and we note that $k_n^0\le k_{n+1}^0$ for all $n$.
Next we choose a non-negative function $\phi$ from $C^{\infty}_c(\R)$, such that 
$\supp(\phi)\subseteq[-1,0]$, and
$\int_{-1}^0\phi(t)\6t=1$. We then define the function
$\tilde{R}_n\colon \R \to[0,\infty)$ as the convolution
\begin{equation}
\tilde{R}_n(t)=n\int_{-1/n}^0k_n^0(t-s)\phi(ns)\6s
=\int_0^1k_n^0(t+\tfrac{u}{n})\phi(-u)\6u, \qquad(t\in\R), 
\label{eq8}
\end{equation}
and we let $R_n$ be the restriction of $\tilde{R}_n$ to $(0,\infty)$. 
Note also that
\[
\tilde{R}_n(t)=n\int_{\R}\phi(n(t-s))k_n^0(s)\6s, \qquad(t\in\R).
\]
Since $k_n^0$ as well as the derivatives of $\phi$ and $\phi$ itself
are all bounded
functions, it follows then by differentiation under the integral sign that
 $\tilde{R}_n$ is a \emph{bounded} $C^{\infty}$-function on $\R$ with
 \emph{bounded} derivatives, and so its restriction $R_n$ to
 $(0,\infty)$ has bounded derivatives too. 
Since $k_n^0$ is decreasing on $(0,\infty)$, it follows
immediately from \eqref{eq8} that so is $R_n$. 
Moreover, $\supp(R_n)\subseteq(0,n]$ by the
definition of $k_n^0$. 

For any $t$ in $(0,\infty)$ and $n$ in $\N$ note next that
\[
R_n(t)\le\int_0^1k_{n+1}^0(t+\tfrac{u}{n})\phi(-u)\6u
\le\int_0^1k_{n+1}^0(t+\tfrac{u}{n+1})\phi(-u)\6u
=R_{n+1}(t).
\]
Moreover, the monotonicity assumptions imply that $k$ is continuous at
almost all $t$ in $(0,\infty)$ (with respect to Lebesgue measure). For
such a $t$ we may further consider $n$ so
large that $t+\frac{u}{n}\in[\frac{1}{n},n]$ for all $u$ in
$[0,1]$. For such $n$ it follows then that
\[
R_n(t)=\int_0^1k(t+\tfrac{u}{n})\phi(-u)\6u
\underset{n\to\infty}{\longrightarrow}
\int_0^1k(t)\phi(-u)\6u=k(t)
\]
by monotone convergence. We conclude that $R_n(t)\nearrow k(t)$ as
$n\to\infty$ for almost all $t$ in $(0,\infty)$. 

Applying the considerations above to the function
$\kappa\colon(0,\infty)\to[0,\infty)$ given by $\kappa(t)=k(-t)$, it
follows that we may construct a sequence 
$(L_n)_{n\in\N}$ of non-negative functions defined on $(-\infty,0)$
and with the following properties:

\begin{itemize}

\item For all $n$ in $\N$ the function $L_n$ has bounded support.

\item For all $n$ in $\N$ we have that $L_n\in C^{\infty}((-\infty,0))$,
  and $L_n^{(p)}$ is bounded for all $p$ in $\N\cup\{0\}$.

\item For all $n$ in $\N$ the function $L_n$ is increasing on
  $(-\infty,0)$.

\item $L_n(t)\nearrow k(t)$ as $n\to\infty$ for almost all $t$ in
  $(-\infty,0)$ (with respect to Lebesgue measure).

\end{itemize}

Next let $\psi(t)=\e^{- t^2}$, and 
note that $\int_{\R}|t|\psi(t)\6t=1$.
We are then ready to define $k_n\colon\R\setminus\{0\}\to[0,\infty)$ by
\[
k_n(t)=
\begin{cases}
a n^2\psi(nt)+R_n(t), &\text{if $t>0$,}\\
a n^2\psi(nt)+L_n(t), &\text{if $t<0$.}
\end{cases}
\]
It is apparent from the argumentation above that $k_n$ satisfies
the conditions (a)-(c) in Section~\ref{compact_case}, and it remains to
show that $\frac{|t|k_n(t)}{1+t^2}\6t\overset{\rm w}{\to} 
a\delta_0+\frac{|t|k(t)}{1+t^2}\6t$ as $n\to\infty$. For any
bounded continuous function $g\colon\R\to\R$ we find that
\begin{equation*}
\begin{split}
\int_{\R}g(t)\frac{|t|k_n(t)}{1+t^2}\6t 
&=an^2\int_{\R}g(t)\frac{|t|\psi(nt)}{1+t^2}\6t 
+\int_{-\infty}^0g(t)\frac{|t|L_n(t)}{1+t^2}\6t
+\int_0^{\infty}g(t)\frac{tR_n(t)}{1+t^2}\6t
\\[.2cm]
&=a\int_{\R}g(\tfrac{u}{n})\frac{|u|\psi(u)}{1+(\frac{u}{n})^2}\6u
+\int_{-\infty}^0g(t)\frac{|t|L_n(t)}{1+t^2}\6t
+\int_0^{\infty}g(t)\frac{tR_n(t)}{1+t^2}\6t
\\[.2cm]
&\underset{n\to\infty}{\longrightarrow}
a\int_{\R}g(0)|u|\psi(u)\6u
+\int_{-\infty}^0g(t)\frac{|t|k(t)}{1+t^2}\6t
+\int_0^{\infty}g(t)\frac{tk(t)}{1+t^2}\6t
\\[.2cm]
&=ag(0)+\int_{\R}g(t)\frac{|t|k(t)}{1+t^2}\6t,
\end{split}
\end{equation*}
where, when letting $n\to\infty$, we used dominated convergence on
each of the three integrals; note in particular that
$\frac{|t|L_n(t)}{1+t^2}$ and $\frac{tR_n(t)}{1+t^2}$ are dominated
almost everywhere by $\frac{|t|k(t)}{1+t^2}$ on the relevant
intervals, and here $\int_{\R}\frac{|t|k(t)}{1+t^2}\6t<\infty$, since
$\frac{k(t)}{|t|}\6t$ is a L\'evy measure. This completes the proof.
\end{proof}

\begin{theorem}\label{main-res}
Any measure $\nu$ in $\CL(\boxplus)$ is unimodal.
\end{theorem}
\begin{proof} We note first that for any probability measure $\mu$ on
  $\R$ and any constant $a$ in $\R$, the free convolution
  $\mu\boxplus\delta_a$ is the translation of $\mu$ by the constant
  $a$, and hence $\mu$ is unimodal, if and only if
  $\mu\boxplus\delta_a$ is unimodal for some (and hence all) $a$ in
  $\R$. For $\boxplus$-infinitely divisible measures this means that
  the measure with free generating pair $(\gamma,\sigma)$ (cf.\
  \eqref{e1.1a}) is unimodal, if and only if the measure with free
  generating pair $(\gamma+a,\sigma)$ is unimodal for some (and
  hence all) $a$ in $\R$. In other words, unimodality depends only
  on the measure $\sigma$ appearing in the free generating pair.

Now let $\nu$ be a measure from $\CL(\boxplus)$ with free
characteristic triplet $(a,\frac{k(t)}{|t|}\6t,\eta)$, where
$a\ge0$, $\eta\in\R$ and $k\colon\R\setminus\{0\}\to[0,\infty)$ is a function as in \eqref{spectral}. According
to the discussion above, it suffices then to show that the measure $\nu^0$
with free generating pair $(0,a\delta_0+\frac{|t|k(t)}{1+t^2}\6t)$ is
unimodal (cf.\ \eqref{ligning3}). By application of
Lemma~\ref{approximation_lemma} we may  
choose a sequence $(k_n)$ of positive functions, satisfying (a)-(c)
in Section~\ref{compact_case}, such that
\begin{equation}
\frac{|t|k_n(t)}{1+t^2}\6t\overset{\rm w}{\longrightarrow}
a\delta_0+\frac{|t|k(t)}{1+t^2}\6t \quad\text{as $n\to\infty$.}
\label{eq7}
\end{equation}
For such $n$ it follows then from Proposition~\ref{unimodality_I} and
\eqref{ligning3} that the measure $\nu_n^0$ with free generating pair
$(0,\frac{|t|k_n(t)}{1+t^2}\6t)$ is unimodal. From \eqref{eq7} and the free
version of Gnedenko's Theorem (cf.\ \eqref{free_Gnedenko_thm})
it follows that $\nu_n^0\overset{\rm w}{\to}\nu^0$ as
$n\to\infty$, and hence \eqref{unimodal_vs_weak_conv} implies that
$\nu^0$ is unimodal, as desired.
\end{proof}

\begin{remark}
A non-degenerate classically selfdecomposable probability
  measure is absolutely continuous with respect to Lebesgue measure
  (see \cite[Theorem~27.13]{Sa}). In the free case it was proved by
  N.~Sakuma (see \cite{S11}) that non-degenerate freely
  selfdecomposable measures 
  have no atoms. By definition (see formula~\ref{UMa}), a unimodal measure 
  does not have a continuous singular part, and via
  Theorem~\ref{main-res} we may thus conclude that 
  also freely selfdecomposable measures are
  absolutely continuous with respect to the Lebesgue measure, unless
  they are degenerate. Moreover, from Huang's density formula 
\cite[Theorem 3.10 (6)] {Hu2}, which is a strengthened version of our
Corollary \ref{density}, one can show that the density 
function of a freely selfdecomposable measure is continuous on $\mathbb{R}$. 
By contrast, the density of a classical selfdecomposable measure may have a
single point of discontinuity (see \cite[Theorem~28.4]{Sa}).

\end{remark}

\subsection*{Acknowledgements}

This paper was initiated during the ``Workshop on Analytic, Stochastic,
and Operator Algebraic Aspects of Noncommutative Distributions and
Free Probability'' at the Fields Institute in July 2013. The authors
would like to express their sincere gratitude for the generous
support and the stimulating environment provided by the Fields
Institute. The authors would also like to thank an anonymous referee
for comments, which have improved the paper, and in particular for pointing out
connections between our paper and Biane's paper \cite{B97}. 

TH was supported by Marie Curie Actions -- International Incoming
Fellowships Project  328112 ICNCP.  

ST was partially supported by The Thiele Centre for Applied
Mathematics in Natural Science at The University of Aarhus.

\vspace{1.5cm}

\begin{minipage}[c]{0.5\textwidth}
Laboratoire de Math\'ematiques \\ 
Universit\'e de Franche-Comt\'e\\
16 route de Gray\\
25030 Besan\c{c}on cedex\\
France

\vspace{2mm}
\noindent
Present address: \\
Department of Mathematics \\
 Hokkaido University \\
   Kita 10, Nishi 8, Kita-ku \\
   Sapporo 060-0810 \\
    Japan\\
{\tt thasebe@math.sci.hokudai.ac.jp}
           
\end{minipage}
\hfill
\begin{minipage}[c]{0.5\textwidth}
Department of Mathematics\\
University of Aarhus\\
Ny Munkegade 118\\
8000 Aarhus C\\
Denmark\\
{\tt steenth@imf.au.dk}
\end{minipage}

\end{document}